\documentclass[preprint]{elsarticle}

\usepackage{amsfonts}
\usepackage{amsmath,amscd,amssymb,amsthm}
\usepackage{mathrsfs}
\usepackage{pstricks}
\usepackage{pst-plot}
\usepackage[cmtip,arrow]{xy}
\usepackage{pb-diagram,pb-xy}
\usepackage{eucal}
\usepackage[left=1.1875in,top=1in,text={6.5in,9in}]{geometry}
\usepackage{syntonly}
\usepackage{url}
\usepackage{verbatim}
\usepackage{graphicx}

\xyoption{all}

\newtheorem{theorem}{Theorem}
\theoremstyle{plain}
\newtheorem{lemma}[theorem]{Lemma}

\newtheorem{cor}[theorem]{Corollary}

\theoremstyle{definition}
\newtheorem{definition}[theorem]{Definition}
\theoremstyle{remark}

\newtheorem{rmk}[theorem]{Remark}

\newcommand{\R}{\mathbb{R}}
\newcommand{\Z}{\mathbb{Z}}

\newcommand{\N}{\mathbb{N}}

\begin{document}


\title{Counting paths in corridors using circular {P}ascal arrays}

\author[vsu]{Shaun V. Ault\corref{cor1}} 
\ead{svault@valdosta.edu}

\author[vsu]{Charles Kicey} 
\ead{ckicey@valdosta.edu}

\cortext[cor1]{Corresponding author}

\address[vsu]{Department of Mathematics and Computer Science, Valdosta
  State University, 1500 N. Patterson St., Valdosta, GA, 31698}

\begin{abstract}
  A circular Pascal array is a periodization of the familiar Pascal's
  triangle.  Using simple operators defined on periodic sequences, we
  find a direct relationship between the ranges of the circular Pascal
  arrays and numbers of certain lattice paths within corridors, which
  are related to Dyck paths. This link provides new, short proofs of
  some nontrivial formulas found in the lattice-path literature.
\end{abstract}


\maketitle

\section{Circular Pascal Arrays and Corridor Paths}

\subsection{Circular Pascal Arrays} 
We begin by defining the {\it circular Pascal arrays} (one for each
integer $d \geq 2$) and explore some of their amazing properties.  By
{\it Pascal array}, we mean something a little more general than the
familiar Pascal's triangle.  In what follows, we interpret the
binomial coefficient $\binom{n}{k}$ as the coefficient of $x^k$ in the
expansion,
\[
  (1+x)^n = \sum_{k\in \Z} \binom{n}{k} x^k,
\]
where we understand $\binom{n}{k} = 0$ if $k < 0$ or $k > n$.  We also
use the convention that $\N$ denotes the set of {\it non-negative
  integers}, $\{0, 1, 2, 3, \ldots\}$.

\begin{definition}
  The {\bf Pascal array} is the array whose row $n$, column $k$ entry
  is equal to $\binom{n}{k}$ where $n \in \N$, $k \in \Z$.
  \[
    \begin{array}{l|lllllllll}
      n \setminus k & \ldots & -1 & 0 & 1 & 2 & 3 & 4 & 5 & \ldots
      \\ \hline 0 & \ldots & 0 & 1 & 0 & 0 & 0 & 0 & 0 & \ldots \\ 1 &
      \ldots & 0 & 1 & 1 & 0 & 0 & 0 & 0 & \ldots \\ 2 & \ldots & 0 &
      1 & 2 & 1 & 0 & 0 & 0 & \ldots \\ 3 & \ldots & 0 & 1 & 3 & 3 & 1
      & 0 & 0 & \ldots \\ 4 & \ldots & 0 & 1 & 4 & 6 & 4 & 1 & 0 &
      \ldots \\ 5 & \ldots & 0 & 1 & 5 & 10 & 10 & 5 & 1 & \ldots
      \\ \vdots && \vdots & \vdots & \vdots & \vdots & \vdots & \vdots
      & \vdots
    \end{array}
  \]
\end{definition}
Recall that for any $n > 0$, entry $(n, k)$ of the array can be found
by the familiar formula,
\begin{equation}\label{eqn.Pascal_formula}
  \binom{n}{k} = \binom{n-1}{k-1} + \binom{n-1}{k}.
\end{equation}

\begin{definition}\label{def.circularPascal}
  Fix an integer $d \geq 2$.  The {\bf circular Pascal array of order
    $d$} is the array whose row $n$, column $k$ entry,
  $\sigma^{(d)}_{n,k}$ (or just $\sigma_{n,k}$ when the context is
  clear), is the (finite) sum:
  \begin{equation}\label{eqn.circular_Pascal_formula}
    \sigma_{n,k} = \sigma^{(d)}_{n,k} = \sum_{j \in \Z} \binom{n}{k +
      dj},
  \end{equation}
  where $n \in \N$ and $k \in \Z$.  
\end{definition}

In what follows, we will be interested in {\bf periodic} sequences and
arrays of numbers.  To be precise, we say a sequence $(a_k)_{k \in
  \Z}$ is {\it periodic}, of period $d$, if $a_{k + dm} = a_k$ for
every $m \in \Z$.  Clearly, for fixed $n \geq 0$, the sequence
$(\sigma^{(d)}_{n,k})_{k \in \Z}$ as defined
in~(\ref{eqn.circular_Pascal_formula}) is periodic of period $d$.
Furthermore, it is easy to see that the entries of the circular Pascal
array satisfy~(\ref{eqn.Pascal_formula}), in the sense that for $n >
0$,
\begin{equation}\label{eqn.recursiveCircularPascal}
  \sigma_{n,k} = \sigma_{n-1,k-1} + \sigma_{n-1,k}.
\end{equation}

For $ d = 5 $, our collection of periodic sequences is indicated below.
\[
  \begin{array}{l|l|lllll|lllll|lllll|l|}
    n \setminus k & \cdots & -5 & -4 & -3 & -2 & -1 & 0 & 1 & 2 & 3 &
    4 & 5 & 6 & 7 & 8 & 9 & \cdots \\ \hline 0 &\cdots & 1 & 0 & 0 & 0
    & 0 & 1 & 0 & 0 & 0 & 0 & 1 & 0 & 0 & 0 & 0 & \cdots \\ 1 &\cdots
    & 1 & 1 & 0 & 0 & 0 & 1 & 1 & 0 & 0 & 0 & 1 & 1 & 0 & 0 & 0 &
    \cdots \\ 2 &\cdots & 1 & 2 & 1 & 0 & 0 & 1 & 2 & 1 & 0 & 0 & 1 &
    2 & 1 & 0 & 0 & \cdots \\ 3 &\cdots & 1 & 3 & 3 & 1 & 0 & 1 & 3 &
    3 & 1 & 0 & 1 & 3 & 3 & 1 & 0 & \cdots \\ 4 &\cdots & 1 & 4 & 6 &
    4 & 1 & 1 & 4 & 6 & 4 & 1 & 1 & 4 & 6 & 4 & 1 & \cdots \\ 5
    &\cdots & 2 & 5 & 10 & 10 & 5 & 2 & 5 & 10 & 10 & 5 & 2 & 5 & 10 &
    10 & 5 & \cdots \\ 6 &\cdots & 7 & 7 & 15 & 20 & 15 &7 & 7 & 15 &
    20 & 15 & 7 & 7 & 15 & 20 & 15 & \cdots \\ 7 &\cdots & 22 & 14 &
    22 & 35 & 35 &22 & 14 & 22 & 35 & 35 & 22 & 14 & 22 & 35 & 35 &
    \cdots \\ 8 &\cdots & 57 & 36 & 36 & 57 & 70 &57 & 36 & 36 & 57 &
    70 & 57 & 36 & 36 & 57 & 70 & \cdots \\ 9 &\cdots & 127 & 93 & 72
    & 93 & 127 &127 & 93 & 72 & 93 & 127 & 127 & 93 & 72 & 93 & 127 &
    \cdots \\ \vdots & \vdots & \vdots & \vdots & \vdots & \vdots &
    \vdots & \vdots & \vdots & \vdots & \vdots & \vdots & \vdots &
    \vdots & \vdots & \vdots & \vdots & \vdots
  \end{array}
\]
Henceforth, we will represent a circular Pascal array by showing only
columns $0, 1, \ldots, d-1$.

While studying a related problem called the Sharing Problem, Charles
Kicey, Katheryn Klimko, and Glen Whitehead\cite{KK2011} noticed that
the circular Pascal array has surprising connections to well-known
sequences for small values of $d$.  Consider the case $d=2$ (shown
below, along with $d=3$ and $d=4$).  Starting in row $n=1$,
$\sigma_{n,k} = 2^{n-1}$ for $k = 0, 1$.  Of course, this reflects a
well-known property of Pascal's triangle: $ \sum_{j\in \Z}
\binom{n}{2j} = \sum_{j \in \Z} \binom{n}{2j+1}$, if $n \geq 1$.
Another way to state this result is to say that the {\it range}
(difference between maximum and minimum values) of the $n^{th}$ row is
$0$ for $n \geq 1$ in the circular Pascal array of order $2$. The
cases $d = 3, 4$ are interesting as well -- for $d=3$, the ranges are
constantly $1$, while for $d=4$, the range of row $n$ is $2^{\lfloor
  n/2 \rfloor}$ -- but the most surprising case is, perhaps, $d=5$.
These ranges form the Fibonacci sequence, as was proved
in~\cite{KK2011}.  Now Fibonacci numbers are no strangers to the
Pascal's triangle; indeed the sequence of diagonal sums, $f_n =
\sum_{j \in \Z} \binom{n-j}{j}$, is easily shown to be the Fibonacci
sequence.  However, this new manifestation of the Fibonacci numbers in
the circular Pascal array of order $5$ was quite unexpected.
\[
  \begin{array}{l|ll|l}
    d=2\\ n \setminus k & 0 & 1 & \textrm{Range}\\ \hline 0 & 1 & 0 &
    1\\ 1 & 1 & 1 & 0\\ 2 & 2 & 2 & 0\\ 3 & 4 & 4 & 0\\ 4 & 8 & 8 &
    0\\ 5 & 16 & 16 & 0\\ 6 & 32 & 32 & 0\\ \vdots & \vdots & \vdots &
    \vdots
  \end{array}
  \quad
  \begin{array}{l|lll|l}
    d=3\\
    n \setminus k &  0 &  1 & 2 & \textrm{Range}\\
    \hline
     0 & 1 &  0  & 0 & 1\\
     1 & 1 &  1  & 0 & 1\\
     2 & 1 &  2 & 1 & 1\\
     3 & 2 &  3 & 3 & 1\\
     4 & 5 & 5 & 6 & 1\\
     5 & 11 &  10 & 11 & 1\\
     6 & 22 &  21 & 21 & 1\\
     \vdots & \vdots & \vdots & \vdots & \vdots
  \end{array}
  \quad
  \begin{array}{l|llll|l}
    d=4 \\
    n \setminus k &  0 &  1 & 2 & 3 & \textrm{Range}\\
    \hline
     0 & 1 &  0  & 0 & 0 & 1\\
     1 & 1 &  1  & 0 & 0 & 1\\
     2 & 1 &  2 & 1 & 0 & 2\\
     3 & 1 &  3 & 3 & 1 & 2\\
     4 & 2 & 4 & 6 & 4 & 4\\
     5 & 6 &  6 & 10 & 10 & 4\\
     6 & 16 &  12 & 16 & 20& 8 \\
     \vdots & \vdots & \vdots &\vdots & \vdots & \vdots
  \end{array}
\]
\[
  \begin{array}{l|lllll|l}
    d=5  \\
    n \setminus k &      0 &  1 &  2 &  3 &  4 & \textrm{Range}\\
    \hline
       0 & 1 &  0 &  0 &  0 &  0 & 1\\
       1 & 1 &  1 &  0 &  0 &  0 & 1\\
       2 & 1 &  2 &  1 &  0 &  0 & 2\\
       3 & 1 &  3 &  3 &  1 &  0 & 3\\
       4 & 1 &  4 &  6 &  4 &  1 & 5\\
       5 & 2 &  5 & 10 & 10 &  5 & 8\\
       6 & 7 &  7 & 15 & 20 &  15 & 13\\
       7 & 22 & 14 & 22 & 35 & 35 & 21\\
       8 & 57 & 36 & 36 & 57 & 70 & 34\\
       9 & 127 & 93 & 72 & 93 & 127 & 55\\
       \vdots & \vdots & \vdots & \vdots & \vdots & \vdots & \vdots 
  \end{array}
\]

So the search began to find other well-known sequences in the
$d$-circular Pascal array.  Another of Charles Kicey's students,
Jonathon Bryant, explored the sequences of ranges for larger values of
$d$ using the OEIS~\cite{OEIS} and found that they were known but in
different contexts.  The OEIS entry for the $d=9$ case (A061551) gives
a tantalizing clue that all of these number sequences are indeed
related.  The main description of A061551 is: ``number of paths along
a corridor width 8, starting from one side,'' and further down the
page, the following note can be found.

\begin{quote}
  Narrower corridors effectively produce A000007, A000012, A016116,
  A000045, A038754, A028495, A030436. An infinitely wide corridor
  (i.e. just one wall) would produce A001405.
\end{quote}

The main result of this paper is to prove that the corridor numbers
are indeed the same as our sequences of ranges of circular Pascal
arrays.  First we define the corridor numbers precisely.


\subsection{Corridor Paths}

For a fixed number $m \geq 0$, the {\bf $m$-corridor} is set of
lattice points $(a, b) \in \N \times \{0, 1, \ldots m\}$.  We will
show that the ranges of circular Pascal arrays of order $m+2$ coincide
with the number of $m$-{\it corridor paths} beginning at the origin.
However, we find it useful to consider paths within an $m$-corridor
that may begin at an arbitrary point, since it makes the arguments no
more difficult and provides a link to a more general formula found in
the combinatorics literature (see \S\ref{sub.Dyck_K-M} below).  To be
precise, a {\bf corridor path} is a path in the $m$-corridor
satisfying the following rules:
\begin{enumerate}
  \item The initial point of the path is at $(0, y_0)$ for some chosen
    $y_0$ with $0 \leq y_0 \leq m$.
  \item The path never leaves the corridor.
  \item Each step in the path is either an up-and-right or
    down-and-right move.
\end{enumerate}
An example corridor path is shown in Fig.~\ref{fig.corridor_path}.

\begin{figure}[hb]
  \begin{center}
    \psset{unit=0.5cm}
    \begin{pspicture}(0,-1)(10,3)
      \psaxes(0,0)(0,0)(10,3)
      \psline[linestyle=dashed](0,2)(10,2)
      \psdots(0,0)(1,1)(2,2)(3,1)(4,2)(5,1)(6,0)(7,1)(8,2)
      \psline(0,0)(1,1)
      \psline(1,1)(2,2)
      \psline(2,2)(3,1)
      \psline(3,1)(4,2)
      \psline(4,2)(5,1)
      \psline(5,1)(6,0)
      \psline(6,0)(7,1)
      \psline(7,1)(8,2)
    \end{pspicture}
  \end{center}
\caption{A path of length $8$ in the $2$-corridor with $y_0 = 0$.}
\label{fig.corridor_path}
\end{figure}
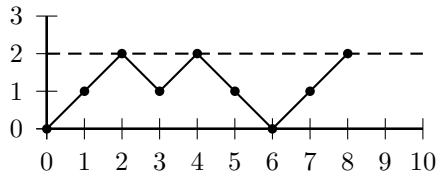

\begin{definition}
  For a fixed $m \geq 0$, the sequence of {\bf corridor numbers of
    order $m$}, $(c^{(m)}_{n})_{n \in \N}$ (or just $(c_{n})$ when the
  context is clear), counts the number of paths of length $n$,
  starting at the origin, in the $m$-corridor.  If we wish to count
  the number of corridor paths starting at $(0, y_0)$, then we may
  write $c^{(m)}_{n, y_0}$ (or $c_{n,y_0}$).
\end{definition}

\begin{rmk}
  Corridor numbers are useful in graph theory as $c^{(m)}_{n}$ counts
  the number of length $n$ paths in the path graph $P_{m+1}$ that
  start at the initial node of the graph, while $c^{(m)}_{n,y_0}$
  counts the number of such paths that start at node $y_0 + 1$.
\end{rmk}

\subsection{Dyck Paths and K-M Paths}\label{sub.Dyck_K-M}

Corridor paths are certain types of {\it lattice paths}.  Indeed, they
may be identified with a variation of {\it Dyck paths}.  Recall, a
Dyck path of order $m$ is a monotonic lattice path from the origin to
the point $(m,m)$ that does not cross the diagonal line $y=x$.  Note
that an order $m$ Dyck path has length $2m$.  For our purposes, we
will assume the path is drawn above the line $y=x$, as in
Fig.~\ref{fig.dyck}. It is well known that the number of Dyck paths of
order $m \geq 0$ is equal to the $m^{th}$ Catalan number,
$C_m$~\cite{Brualdi}.  We will consider the following variation of
Dyck paths found in Krattenthaler-Mohanty~\cite{KraMoh}:

\begin{definition}\label{def.KM}
  Let $s, t \in \Z$ such that $t \geq 0 \geq s$ and $a, b \in \Z$ such
  that $a+t \geq b \geq a+s$.  A {\bf K-M path} is a monotonic lattice
  path from the origin to $(a, b)$ that does not cross either of the
  lines $y=x+s$ or $y = x+t$.  The number of such K-M paths is denoted
  $D(a, b; s, t)$.  If $b > a+t$ or $b < a+s$, then define $D(a, b; s,
  t) = 0$.
\end{definition}

There is an affine transformation taking K-M paths to corridor paths.
With $a, b, s, t$ as in Definition~\ref{def.KM}, map the point $(a, b)
\mapsto (a+b, b-a-s)$.  Then the line $y=x+s$ maps to the $x$-axis,
the line $y=x+t$ maps to the line $y = t-s$, and the origin (the
initial point for all K-M paths) maps to $(0, -s)$.  Thus, with $m =
t-s$ and $y_0 = -s$, the result is the $m$-corridor with initial point
$(0, y_0)$.  When $s = 0$ and $t = m$, we recover the corridor numbers
of order $m$:
\begin{equation}
  c_n = \sum_{a+b=n} D(a,b; 0, m).
\end{equation}
That is, the number of length $n$ paths in the $m$-corridor is equal
to the number of K-M paths of length $n$, lying between $y=x$ and
$y=x+m$ (Compare Fig.~\ref{fig.restricted_dyck} and
Fig.~\ref{fig.corridor_path}, for example).

\begin{figure}[!ht]
  \begin{center}
    \psset{unit=0.5cm}
    \begin{pspicture}(0,-1)(6,6)
      \psaxes(0,0)(0,0)(6,6) \psline[linestyle=dashed](0,0)(6,6)
      \psdots(0,0)(0,1)(1,1)(1,2)(1,3)(1,4)(2,4)(2,5)(3,5)(4,5)(5,5)
      \psline(0,0)(0,1) \psline(0,1)(1,1) \psline(1,1)(1,2)
      \psline(1,2)(1,3) \psline(1,3)(1,4) \psline(1,4)(2,4)
      \psline(2,4)(2,5) \psline(2,5)(3,5) \psline(3,5)(4,5)
      \psline(4,5)(5,5)
    \end{pspicture}
  \end{center}
\caption{A Dyck path of order $5$.}
\label{fig.dyck}
\end{figure}
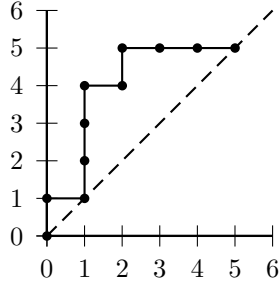

\begin{figure}[!ht]
  \begin{center}
    \psset{unit=0.5cm}
    \begin{pspicture}(0,-1)(5,6)
      \psaxes(0,0)(0,0)(5,6) 
      \psline[linestyle=dashed](0,0)(6,6)
      \psline[linestyle=dashed](0,2)(4,6)
      \psdots(0,0)(0,1)(0,2)(1,2)(1,3)(2,3)(3,3)(3,4)(3,5)
      \psline(0,0)(0,1) \psline(0,1)(0,2) \psline(0,2)(1,2)
      \psline(1,2)(1,3) \psline(1,3)(2,3) \psline(2,3)(3,3)
      \psline(3,3)(3,4) \psline(3,4)(3,5)
    \end{pspicture}
  \end{center}
\caption{A K-M path of length 8 ending at (3,5)}
\label{fig.restricted_dyck}
\end{figure}
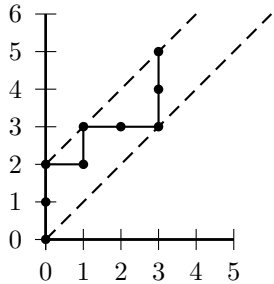

These types of lattice paths and other variations have been studied
extensively in the combinatorics literature (see, for
example,~\cite{KraMoh,Mohanty,Narayana}).  Krattenthaler and Mohanty
give a formula involving sums of binomial coefficients which looks
quite related to our definition of the circular Pascal
array~(\ref{def.circularPascal}).
\begin{equation}\label{eqn.KraMoh_formula}
  D(a,b;s, t) = \sum_{k \in \Z} \left( \binom{a+b}{a-k(t-s+2)} -
  \binom{a+b}{a-k(t-s+2) + t + 1}\right).
\end{equation}
\begin{rmk}
  This formula goes back to~\cite{Mohanty}, \S1.3, Thm.~2. Also
  see~\cite{Takacs}, and {\it Bertrand's ballot problem}.
\end{rmk}

In the next section, we prove the connection between circular Pascal
arrays and corridor numbers, and this will lead to a simpler
derivation of~(\ref{eqn.KraMoh_formula}) than the proofs currently
found in the literature, to our knowledge.


\section{The Main Result}

\subsection{Shift and Difference Operators}
Fix an integer $d \geq 2$.  In what follows we work in the vector
space $ \R^{\infty} $ of real-valued sequences indexed by the set of
integers $ \Z $.  All of our vectors will be periodic (in particular,
with period $d$ or $2d$).  Let $ I $ be the identity operator on
$\R^{\infty}$.  For $ {\bf{x}} = ( x_k )_{k \in \Z} \in \R^{\infty}$,
we will make use of the right shift operator defined by $R ({\bf{x}})
= ( x_{k-1} )_{k \in \Z} $.  Let $ L = R^{-1} $. Note that any powers
or $ L $ and $ R $ would, of course, commute.  Let us introduce a
difference operator $D$ on $ \R^{\infty} $ by $ D = I - L $ (a
negative of a ``discrete deriviative"). The role of $D$ will become
clear as we search for minimum and maximum values in our Pascal
arrays.  Define a periodic ``unit vector,'' ${\bf e_{0}}$, by:
\begin{equation}\label{eqn.e}
  {\bf e}_0 = (\dots, 0, 1, \underbrace{0, \ldots, 0}_{d-1
    \;\textrm{zeros}}, {\bf{1}}, \underbrace{0 , \dots, 0}_{d-1
    \;\textrm{zeros}}, 1, 0, \ldots ).
\end{equation}

From these building blocks alone, we can analyze the circular Pascal
array as well as our corridor numbers.  Note that the recursive
formula~(\ref{eqn.recursiveCircularPascal}) can be encoded by the
operator $I + R$.  In particular, if $\boldsymbol\sigma_n =
(\sigma_{n,k} )_{k \in \Z}$ is the $n^{th}$ row of the circular Pascal
array of order $d$, then for $n \geq 1$, we have:
\begin{equation}\label{eqn.vectorPascal}
  \boldsymbol\sigma_n = (I + R)\boldsymbol\sigma_{n-1} = (I +
  R)^n\boldsymbol\sigma_{0},
\end{equation}
with the initial vector $\boldsymbol\sigma_{0} = {\bf e}_0$.  However,
we will find it useful to allow more general initial vectors
$\boldsymbol\sigma_0$ in the analysis below.

\begin{rmk}
  Here and throughout, as an aid to the reader, we will bold-face the
  entry corresponding to $k = 0$ for numerical vectors in $
  \R^{\infty} $ (unless it is clear from context).
\end{rmk}

\subsection{Up-Sampling}

We also introduce a variation of the {\it up-sample} operator of
digital signal processing~\cite{Strang}.  For $ {\bf x} = ( x_k )_{k
  \in \Z}$, we will define our up-sample operator $U$ as follows:
\[ 
  U ({\bf{x}}) = ( x_{\lfloor k/2 \rfloor} )_{k \in \Z} = ( \dots,
  x_{-1}, x_{-1}, x_{0} , x_{0} , x_1 , x_1, \dots ) \ .
\]

\begin{definition}
  The {\bf{up-sampled circular Pascal array of order $d$}}, denoted $
  ( {\bf{p}}_n )_{n \ge 0} $ is defined by ${\bf p}_n = U(
  \boldsymbol\sigma_n ) {\textrm{ for all }} n \geq 0$.
\end{definition}
Note that up-sampled array of order $d$ repeats in blocks of size
$2d$.  Observe that $ U(I + R) = (I + R^2)U $ on $ \R^{\infty} $, so
it follows that the up-sampled array can be expressed inductively as
\begin{equation}\label{eqn.p_n_inductive}
  {\bf p}_n = (I + R^2)^n {\bf p}_0,
\end{equation}
where ${\bf p}_0$ is the initial vector.  In our applications, we
typically use ${\bf p}_0 = U(\boldsymbol\sigma_0)$, for our chosen
initial vector $\boldsymbol\sigma_0$, and so in the prototypical case
({\it i.e.}, $\boldsymbol\sigma_0 = {\bf e}_0$), we have
\begin{equation}\label{eqn.p_0}
  {\bf p}_0 = ( \dots, {\bf{1}}, 1, \underbrace{0, 0, \dots , 0,
    0}_{2d-2 \;\textrm{zeros}}, 1, 1, 0, \dots ).
\end{equation}
Now define the differences: $ ( {\bf{q}}_n )_{n \geq 0} $ , by $ {\bf
  q}_n = D( {\bf p}_n ) $ for all $ n \geq 0 $. Each row in this
collection of differences is of course $2d$-periodic.  It turns out
that $ {\bf q}_{n+1} $ can be obtained from $ {\bf q}_{n} $ the same
way as $ {\bf p}_{n+1} $ is obtained from $ {\bf p}_{n} $:
\begin{lemma}\label{lemma.main}
  For all $n \geq 0$, ${\bf q}_{n} = (I + R^2)^n {\bf q}_0$.
\end{lemma}
\begin{proof}
  Because $ I + R^2 $ and $ D = I - L $ commute,
  we have, for $n \geq 1$,
  \[
    {\bf q}_{n} = D ( {\bf p}_{n} ) = D ( I + R^2)( {\bf p}_{n-1} ) =
    ( I + R^2) D ( {\bf p}_{n-1} ) = ( I + R^2){\bf q}_{n-1}.
  \]
  A simple induction gives the result.
\end{proof}

Note that in the prototypical case, when ${\bf p}_0$ is given
by~(\ref{eqn.p_0}), we have
\[
  {\bf q}_0 = ( \dots, 0, -1, {\bf{0}}, 1, \underbrace{0, 0, \dots ,
    0}_{2d-3 \;\textrm{zeros}}, -1, 0, 1, 0, \dots ).
\]
However, let us analyze a more general case that will become useful in
later sections.  Let $0 \leq y_0 \leq d - 2$ and define the initial
vector $\boldsymbol\sigma_0$ via
\begin{equation}\label{eqn.startPascal}
  \boldsymbol\sigma_0 = \left( I + R + R^2 + \cdots +
  R^{y_0}\right){\bf e}_0 = ( \dots , \underbrace{ {\bf{1}}, 1, \dots,
    1}_{\text{$y_0 + 1$ ones}} , \underbrace{0, 0, \dots ,
    0}_{\text{$d - (y_0 + 1)$ zeros}} , \dots ),
\end{equation}
and define $\boldsymbol\sigma_n$ inductively
by~(\ref{eqn.vectorPascal}).  For the up-sampled version we set ${\bf
  p}_n = U(\boldsymbol\sigma_n)$ for all $n \geq 0$.  In particular,
for $ n=0$,
\begin{equation}\label{eqn.start}
  {\bf p}_0 = (I + R + R^2 + \cdots +
  R^{2y_0 + 1}){\bf e}'_0 = ( \dots, \underbrace{ {\bf{1}}, 1, \dots
    1, 1}_{\text{$2 y_0 + 2$ ones}} , \underbrace{0, 0, \dots ,0,
    0}_{\text{$2d - 2y_0 - 2$ zeros}}, \ldots ),
\end{equation}
where ${\bf e}'_0$ is the $2d$-periodic analog of ${\bf e}_0$.  The
relationship between up-sampled Pascal arrays and corridor number
rests on the following key fact.
\begin{lemma}\label{lem.q_0_formula}
  With the general initial vector ${\bf p}_0$ defined
  by~(\ref{eqn.start}), $L^{y_0}{\bf q}_0 = (-L^{y_0+1} +
  R^{y_0+1}){\bf e}'_0$.
\end{lemma}
\begin{proof} 
  \begin{eqnarray*}
    {\bf q}_0 &=& D({\bf p}_0) \\ &=& (I - R^{-1})(I + R + R^2 +
    \cdots + R^{2y_0+1}){\bf e}'_0\\ &=& (-R^{-1} + R^{2y_0 + 1}){\bf
      e}'_0\\ &=& (-L + R^{2y_0 + 1}){\bf e}'_0 \\ L^{y_0}{\bf q}_0
    &=& (-L^{y_0+1} + R^{y_0+1}){\bf e}'_0.
  \end{eqnarray*}
\end{proof}

\subsection{Dual Corridors}
Let us shift our attention now to corridors and corridor numbers.
Again fix $ d \geq 2$. Recall that the corridor numbers of order $m =
d-2$ count corridor paths starting at $ (0,0) $ in $ \N \times \{ 0,
1, \dots , d-2 \} $.  Instead, let us shift up a unit and consider
corridor paths starting at $ (0,1) $ in $ \N \times \{ 1, 2, \dots , d
- 1 \} $. In fact, we introduce a dual corridor structure, of {\it
  positive} corridor paths starting at $ (0,1) $ in $ \N \times \{ 1,
2, \dots , d - 1 \} $ together with {\it negative} corridor paths
starting at $ (0,-1) $ in $ \N \times \{ -1, -2, \dots , -(d - 1) \}
$.  Embed the dual corridors into the lattice $\N \times \Z$ and
extend by $2d$-periodicity in the second component.  Let us denote the
signed number of paths incoming to vertex $ (n,k) $ by $ v_{n,k}
$. Here the signs are chosen according to $ v_{n, k} > 0 $ if $ k
\equiv 1, 2, \dots , d-1 $ (mod $2d$) and $ v_{n, k} < 0 $ if $ k
\equiv -1, -2, \dots , -(d-1) $ (mod $2d$).  Let $ {\bf v}_n = (
v_{n,k} )_{k \in \Z} $, which one may say is the {\bf state} of our
periodic corridor at step $n$.  More generally, consider the initial
state, ${\bf v}_0$, defined so that it will correspond to a corridor
path starting at vertex $(0, y_0+1)$, {\it i.e.},
\begin{equation}\label{eqn.v_0}
  {\bf v}_0 = ( \dots, 0, -1, \underbrace{ 0, \ldots, 0, {\bf 0}, 0,
    \ldots, 0} _{\text{$2 y_0 + 1$ zeros}} , 1, 0, \ldots).
\end{equation}
The key observation, using Lemma~\ref{lem.q_0_formula}, is that there
is a link between differences of circular Pascal array entries and the
vertex state numbers:
\begin{equation}\label{eqn.v_0_L}
  {\bf v}_0 = (-L^{y_0+1} + R^{y_0+1}){\bf e}'_0 = L^{y_0} { \bf q}_0.
\end{equation}

\begin{figure}[ht]
  \begin{center}
    \psset{unit=1cm,xunit=1.5cm}
    \begin{pspicture}(0,-5)(5,5)
      \psaxes(0,0)(0,-5)(5,5)
      \psline[linestyle=dashed,linecolor=gray](0,1)(5,1)
      \psline[linestyle=dashed,linecolor=gray](0,-1)(5,-1)
      \psline[linestyle=dashed,linecolor=gray](0,4)(5,4)
      \psline[linestyle=dashed,linecolor=gray](0,-4)(5,-4)
      \psline[linestyle=dashed,linecolor=gray](0,5)(5,5)
      \psline[linestyle=dashed,linecolor=gray](0,-5)(5,-5)
      \psline(0,1)(1,2)
      \psline(1,2)(2,3)
      \psline(1,2)(2,1)
      \psline(2,1)(3,2)
      \psline(2,3)(3,2)
      \psline(2,3)(3,4)
      \psline[doubleline=true](3,2)(4,1)
      \psline[doubleline=true](3,2)(4,3)
      \psline(3,4)(4,3)
      \psline[doubleline=true](4,1)(5,2)
      \psline(4,3)(5,2)
      \psline(4,3.07)(5,2.07)
      \psline(4,2.93)(5,1.93)
      \psline(4,3)(5,4)
      \psline(4,3.07)(5,4.07)
      \psline(4,2.93)(5,3.93)
      \psline(0,-1)(1,-2)
      \psline(1,-2)(2,-3)
      \psline(1,-2)(2,-1)
      \psline(2,-1)(3,-2)
      \psline(2,-3)(3,-2)
      \psline(2,-3)(3,-4)
      \psline[doubleline=true](3,-2)(4,-1)
      \psline[doubleline=true](3,-2)(4,-3)
      \psline(3,-4)(4,-3)
      \psline[doubleline=true](4,-1)(5,-2)
      \psline(4,-3)(5,-2)
      \psline(4,-3.07)(5,-2.07)
      \psline(4,-2.93)(5,-1.93)
      \psline(4,-3)(5,-4)
      \psline(4,-3.07)(5,-4.07)
      \psline(4,-2.93)(5,-3.93)
      \psdots(0,1)(1,2)(2,1)(2,3)(3,2)(3,4)
        (4,1)(4,3)(5,2)(5,4)
      \psdots(0,-1)(1,-2)(2,-1)(2,-3)(3,-2)(3,-4)
        (4,-1)(4,-3)(5,-2)(5,-4)
      \uput[70](0,1){$1$}
      \uput[u](1,2){$1$}
      \uput[u](2,1){$1$}
      \uput[u](2,3){$1$}
      \uput[u](3,2){$2$}
      \uput[u](3,4){$1$}
      \uput[u](4,1){$2$}
      \uput[u](4,3){$3$}
      \uput[u](5,2){$5$}
      \uput[u](5,4){$3$}
      \uput[-80](0,-1){$-1$}
      \uput[d](1,-2){$-1$}
      \uput[d](2,-1){$-1$}
      \uput[d](2,-3){$-1$}
      \uput[d](3,-2){$-2$}
      \uput[d](3,-4){$-1$}
      \uput[d](4,-1){$-2$}
      \uput[d](4,-3){$-3$}
      \uput[d](5,-2){$-5$}
      \uput[d](5,-4){$-3$}
      \psline[linestyle=dotted](0,1)(1,0)
      \psline[linestyle=dotted](2,1)(3,0)
      \psline[linestyle=dotted](4,1)(5,0)
      \psline[linestyle=dotted](0,-1)(1,0)
      \psline[linestyle=dotted](2,-1)(3,0)
      \psline[linestyle=dotted](4,-1)(5,0)
      \uput[u](1,0){$0$}
      \uput[u](3,0){$0$}
      \uput[u](5,0){$0$}
    \end{pspicture}
  \end{center}
\caption{The dual corridor structure for $d=5$ and $y_0 = 0$ For
  example, the initial state is ${\bf v}_0 = ( \ldots, 0, -1, {\bf 0},
  1, 0, \ldots )$, and we have ${\bf v}_5 = ( \ldots, 0, -3, 0, -5, 0,
  {\bf 0}, 0, 5, 0, 3, 0, \ldots)$.}
\label{fig.dual_corridor}
\end{figure}
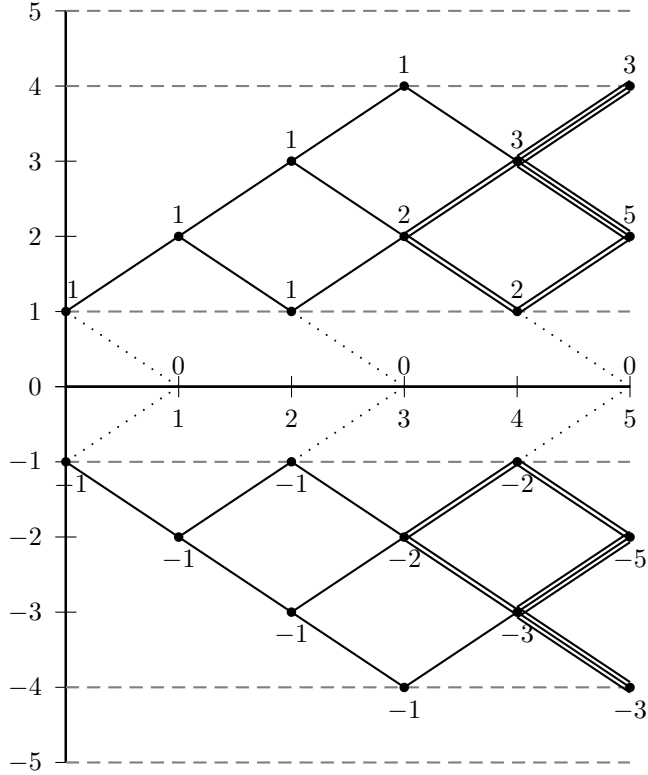
  
By the very definition of the $ {\bf v}_n $ we have
\begin{equation}\label{eqn.corridor1}
  v_{n,0} = v_{n, \pm d } = v_{n, \pm 2d } = \dots = 0.
\end{equation}
Moreover, $ v_{n,k} $ are antisymmetric about $ k = 0 $, {\it i.e.},
\begin{equation}\label{eqn.corridor2}
  v_{n,-k} = -v_{n, k } \textrm{ for all } k \in \Z.
\end{equation}

Now consider the state of our corridor at step $ n + 1 $.  Recall that
the upper and lower corridor states are represented by $ v_{n+1,k} $
for $ k = 1, 2, \dots, d-1 $ and $ k = -1, -2, \dots, -(d-1) $
respectively.  An interior vertex $ (n,k) $, $k = 2, \dots, d - 2 $
receives paths from both $ (n, k-1) $ and $ (n, k+1) $, and so
\begin{equation}\label{eqn.corridor3}
  v_{n+1, k} = v_{n,k-1} + v_{n,k+1}.
\end{equation}
At the boundaries, however, $ v_{n+1, 1} = v_{n,2} $ and $ v_{n+1,
  d-1} = v_{n,d-2} $.  But observe that~(\ref{eqn.corridor1})
and~(\ref{eqn.corridor2}) imply that~(\ref{eqn.corridor3}) does hold
for all $k = -(d-1), -(d-2), \dots, 0, 1, \dots d $.  Then by
periodicity,~(\ref{eqn.corridor3}) holds for all $ k \in \Z $.  In our
operator notation, this gives
\begin{equation}\label{eqn.corridor4}
  {\bf v}_{n + 1} = (L + R) {\bf v}_{n }.
\end{equation}

\subsection{Proof of the Main Theorem}

We are now in position to prove the main result of this paper.

\begin{lemma}\label{lem.main}
  Fix $d \geq 2$.  Let $0 \leq y_0 \leq d-2$.  For each $n \geq 0$,
  ${\bf v}_n = L^{n + y_0 } {\bf q}_n$, {\it i.e.}, $ v_{n,k} =
  q_{n,k+n + y_0 }$, for all $k \in \Z$.
\end{lemma}
\begin{proof}
  By~(\ref{eqn.corridor4}) and induction, $ {\bf v}_{n} = (L + R)^n
  {\bf v}_{0} $ for all $ n \geq 0 $.  But $ (L + R)^n = [ L (I +
    R^2)]^n = L^n (I + R^2)^n $. 
  Using $ {\bf v}_0
  = L^{y_0}{\bf q}_0 $ and
  Lemma~\ref{lemma.main}, we have
  \[
    {\bf v}_n = L^n(I+R^2)^n(L^{y_0}{\bf q}_0) =
    L^{n+y_0}(I+R^2)^n{\bf q}_0 = L^{n+y_0}{\bf q}_n.
  \]
\end{proof}

\begin{theorem}\label{thm.main}
  For fixed $ d \geq 2 $, and $0 \leq y_0 \leq d-2$, the $ n^{th} $
  corridor number of order $ d - 2 $, with paths beginning at $(0,
  y_0)$ is given by $ c^{(d-2)}_{n,y_0} = p^{(d)}_{n,n+y_0} -
  p^{(d)}_{n,n+y_0+d}$.  Moreover, $p^{(d)}_{n,n+y_0} $ and $
  p^{(d)}_{n,n+y_0+d}$ are respectively, the maximum and minimum
  values found in the $ n^{th} $ row of the circular Pascal array of
  order $d$ whose initial vector, $\boldsymbol\sigma_0$ is given
  by~(\ref {eqn.startPascal}).
\end{theorem}

\begin{proof}
  Let $p_{n,k}$, {\it resp.}~$q_{n,k}$, be the $k^{th}$ entry of ${\bf
    p}_n$, {\it resp.}~${\bf q}_n$.  By definition, $ c_{n, y_0} =
  \sum_{k = 0}^{d - 1} v_{n,k} $. So by Lemma~\ref{lem.main},
  \[
    c_{n, y_0} = \sum_{k = 0}^{d - 1} q_{n,k+n+y_0} = \sum_{k =
      n+y_0}^{n + y_0 + d - 1} q_{n,k} = \sum_{k = n+y_0}^{n + y_0+ d
      - 1} \left( p_{n,k} - p_{n,k+1} \right) = p_{n,n+y_0} - p_{n, n+
      y_0+d}.
  \]
  Moreover, by the definition of the dual corridor states,
  $\mathbf{v}_n$, it is clear that the differences $ q_{n,k}$ satisfy:
  \[
    q_{n,n+y_0 + j} = \left\{ \begin{array}{ll}
      v_{n,0} = 0, & \textrm{if $j\equiv 0 \quad (\textrm{mod}\; 2d)$,}\\
      v_{n,j} \geq 0, & \textrm{if $j \equiv 1, 2, \ldots, d-1
      \quad (\textrm{mod}\; 2d)$,} \\
      v_{n,d} = 0, & \textrm{if $j \equiv d \quad (\textrm{mod}\; 2d)$,}\\
      v_{n,j} \leq 0, & \textrm{if $j \equiv d+1, d+2, 
        \ldots, 2d - 1 \quad (\textrm{mod}\; 2d)$}.
    \end{array}\right.
  \]
  Now since $q_{n,n+y_0+j} = p_{n,n+y_0+j}-p_{n,n+y_0+j+1}$ is
  essentially a negative discrete derivative, it follows that
  $p_{n,n+y_0}$ is a maximum value of ${\bf p}_n$, and $p_{n,n+y_0+d}$
  is a minimum value of ${\bf p}_n$.  Of course the maximum and
  minimum values of the up-sampled array correspond to those in the
  original $ ( \boldsymbol\sigma_n )_{ n \in \N }$, which completes
  the proof.
\end{proof}

{\it Example.} Let $ d = 8 $ and $y_0 = 2$. To find the corridor
numbers of order $d - 2 = 6$ in $ \N \times \{ 0, 1, \dots , 6 \} $,
but with corridor paths starting at $ (0, 2) $, we take the range of
the $ n^{th} $ of the Pascal array mod 8, but using
$\boldsymbol\sigma_0 = (1, 1, 1, 0, 0, 0, 0, 0 ) \in \R^8 $ (extended
periodically) as our initial row (See Fig.~\ref{fig.circular_d=8}).

\begin{figure}
\[
  \begin{array}{l|llllllll|l}
    d=8 \\ 
    n \setminus k & 0 & 1 & 2 & 3 & 4 & 5 & 6 & 7 &
    \textrm{Range} \\ 
    \hline 
    0 & 1 & 1 & 1 & 0 & 0 & 0 & 0 & 0 & 1\\ 
    1 & 1 & 2 & 2 & 1 & 0 & 0 & 0 & 0 & 2\\ 
    2 & 1 & 3 & 4 & 3 & 1 & 0 & 0 & 0 & 4\\
    3 & 1 & 4 & 7 & 7 & 4 & 1 & 0 & 0 & 7\\ 
    4 & 1 & 5 & 11 & 14 & 11 & 5 & 1 & 0 & 14\\
    5 & 1 & 6 & 16 & 25 & 25 & 16 & 6 & 1 & 24\\
    6 & 2 & 7 & 22 & 41 & 50 & 41 & 22 & 7 & 48\\
    7 & 9 & 9 & 29 & 63 & 91 & 91 & 63 & 29 & 82\\ 
    8 & 38 & 18 & 38 & 92 & 154 & 182 & 154 & 92 & 164\\
    9 & 130 & 56 & 56 & 130 & 246 & 336 & 336 & 246 & 280\\
    \vdots & \vdots & \vdots & \vdots & \vdots & \vdots 
    & \vdots & \vdots & \vdots & \vdots
  \end{array}
\]
\caption{Circular Pascal numbers corresponding to $d=8$ and $y_0=2$.}
\label{fig.circular_d=8}
\end{figure}

\begin{cor}\label{cor.main}
  The $(d-2)$-corridor number $c_n$ ({\it i.e.}, when $y_0 = 0$)
  equals the range of the $n^{th}$ row of the (standard) circular
  Pascal array or order $d$.  More explicitly, $c^{(d-2)}_n =
  \sigma^{(d)}_{n, \lfloor n/2 \rfloor} - \sigma^{(d)}_{n, \lfloor
    (n+d)/2 \rfloor}$.
\end{cor}


\section{Further Results}

\subsection{Deriving the K-M Formula}

In this section we prove the K-M formula~(\ref{eqn.KraMoh_formula}) in
full generality. Let $s, t \in \Z$ such that $t \geq 0 \geq s$, and
$a, b \in \Z$ such that $a + t \geq b \geq a + s$.  Then the number of
K-M paths, $D(a,b; s, t) $ starting at $ (0,0) $ and ending at $(a,b)$
is equivalent to the number of $m$-corridor paths in $ \N \times \{ 1,
2, \dots, d - 1 \} $ beginning at $(0, y_0+1) $ and ending at $ (a +
b, b - a + y_0 + 1)$, where $m = d - 2 = t - s$ and $y_0 = - s $.
That is, $D(a,b; s, t) = v_{ a+b,\, b - a + y_0 + 1}$.  Using
Lemma~\ref{lem.main} and the definitions of ${\bf q}_n$, ${\bf p}_n$,
and $\boldsymbol\sigma_n$, this leads to
\begin{eqnarray*}
  D(a,b; s, t) &=& q_{a+b,\, (b-a+y_0+1)+(a+b)+y_0} \\
  &=& q_{a+b,\, 2b - 2s + 1}\\
  &=& p_{a+b,\, 2b-2s+1} - p_{a+b,\, 2b-2s+2} \\
  &=& \sigma_{a+b,\, b-s} - \sigma_{a+b,\, b-s+1}.
\end{eqnarray*}
Using the linearity of $ (I + R)^n $ and the fact that $ R $ commutes
with $ (I + R)^n $, then it is easy to see that
\begin{equation}\label{eqn.modd2}
  \sigma_{n,k} = \sum_{j \in \Z} \left[ { \binom{n}{k + dj}} + { \binom{n}
      {(k - 1) + dj}} + \dots + { \binom{n}{(k - y_0) + dj}} \right].
\end{equation}
Then by~(\ref{eqn.modd2}), we have:
\begin{eqnarray*} 
  D(a,b; s, t) & = & \sigma_{a+b,\, b-s} - \sigma_{a+b, b-s + 1}
  \\ &=& \sum_{k \in \Z} \left[ { \binom{a + b}{b-s+dk} } + { \binom{a
        + b}{b-s-1+dk } } + \cdots + { \binom{a + b}{ b-s-(-s) + dk} }
    \right] \\ && - \sum_{k \in \Z} \left[ { \binom{a + b}{b-s+1+dk} }
    + { \binom{a + b}{b-s+dk } } + \cdots + { \binom{a + b}{
        b-s+1-(-s) + dk} } \right]\\ & = & \sum_{k \in \Z} \left[ {
      \binom{a + b}{ b + dk } } - { \binom{a + b}{ b - s + 1 + dk } }
    \right]\\ &=& \sum_{k \in \Z} \left[ { \binom{a + b}{ a - dk } } -
    { \binom{a + b}{ a + s - 1 - dk } } \right]\\ &=& \sum_{k \in \Z}
  \left[ { \binom{a + b}{ a - k(t-s+2)} } - { \binom{a + b}{ a -
        k(t-s+2) + s - 1} } \right].
\end{eqnarray*}
The final step is to re-index the second terms via $k \mapsto k+1$,
and obtain~(\ref{eqn.KraMoh_formula}).

\subsection{Infinite Width Corridors}
Consider corridor paths beginning at $ (0,y_0) $, but in the infinite
corridor $ \N \times \N $; we denote the number of such paths of
length $n$ by $ c^{ ( \infty ) }_{n , y_0} $.  It is easy to see that
$ c^{ ( \infty ) }_{n , y_0} $ is the number of $ n $-tuples $ ( r_1 ,
r_2 , \dots , r_n ) $ satisfying $ r_k \in \{ -1 , 1 \} $ and $
\sum_{j = 1}^{k} r_j \ge -y_0 $ for all $ k = 1, 2, \dots , n $. We
pass on the challenge found in~\cite{Leeuwen2010}.

\begin{quote}
  The most basic case of the enumerative coincidences that we shall
  study is the fact that there are $\binom{2n}{n}$ positive walks of
  length $2n$, a number that also (and more obviously) counts the
  recurrent walks of that length. This result appears to be well
  known, at least in the lattice path community, but in view of its
  simplicity it is somewhat surprising that it does not receive
  prominent mention in the enumerative combinatorics literature. We do
  not know whether any nice bijective proofs for this result are
  known, but it would at least seem that none are ``well known.''
\end{quote}

Although we do not have a {\it bijective} proof, we now provide a
simple proof based on our structure already in place.  Fix $ n \geq 0
$ and $ y_0 \geq 0 $. For $ m = n + y_0$, {\it i.e., $d=n+y_0+2$},
there is no difference between $ c^{ ( \infty ) }_{n , y_0 } $ and $
c^{(m)}_{n , y_0 } $, since the paths of length $n$ in $ \N \times \{0
, 1, \dots , m \} $ are not yet restricted by the upper wall.

Thus by Theorem~\ref{thm.main},
\begin{eqnarray*}
  c^{ ( \infty ) }_{n , y_0 } = c^{ ( n + y_0) }_{n , y_0 } & = &
  p^{(n+y_0+2)}_{n,\, n + y_0} - p^{(n+y_0+2)}_{n, \,n + y_0 +
    (n+y_0+2)} \\ & = & p^{(n+y_0+2)}_{n,\, n + y_0} -
  p^{(n+y_0+2)}_{n, \,2n + 2y_0 + 2} \\ & = & \sigma^{(n+y_0+2)}_{n,\,
    \lfloor (n + y_0)/2 \rfloor } - \sigma^{(n+y_0+2)}_{n,\, n + y_0 +
    1}.
\end{eqnarray*}
By~(\ref{eqn.modd2}) we see that the second term must be zero and so $
c^{ ( \infty ) }_{n , y_0 } = \sigma_{n, \lfloor (n + y_0)/2 \rfloor
}$.  Finally, in the case $ y_0 = 0 $, we have $ c^{ ( \infty ) }_{n ,
  0 } = c^{ (n) }_{n} = \sigma^{(n+2)}_{n, \lfloor n/2 \rfloor} -
\sigma^{(n+2)}_{n, n+1} = \binom{n}{\lfloor n/2 \rfloor} - 0 $, the
central binomial coefficient as expected.

{\it Example.} Let $ n = 4 $ and $y_0 = 2$. To find the infinite
corridor number, we have $ c^{ ( \infty ) }_{4 , 2 } = c^{ ( 6 ) }_{4
  , 2 } = p_{4,6} - p_{4,14} = \sigma^{(8)}_{4,3} - \sigma^{(8)}_{4,
  7} $, which may be computed using~(\ref{eqn.modd2}), but actually
appear in Fig.~\ref{fig.circular_d=8}.  Indeed, we have $ c^{ ( \infty
  ) }_{4 , 2 } = \sigma_{4,3} - \sigma_{4, 7} = 14-0 = 14$.

\subsection{Three-choice Corridors}
When the allowable moves in a corridor include remaining at the same
level, the paths are often called ``Motzkin.''  The structure we have
already set up extends easily to such Motzkin paths.  Fix $ d \geq 2
$. We define an (up-sampled) circular Pascal-type array $ ( {\bf{p}}_n
)_{n \in \N} $ still using $ {\bf{p}}_0 = {\bf{e}}'_0 + R( {\bf{e}}'_0
) $ defined above, but now using $ T = I + R + R^2 $ to transition
from $ {\bf{p}}_n $ to $ {\bf{p}}_{n+1} $. Define the difference array
$ ( {\bf{q}}_n )_{n \in \N} $ as above, $ {\bf{q}}_n = D ( {\bf{p}}_n
) $ where $ D = I - L $. 

\begin{rmk}
  If we had begun with $ {\bf{p}}_0 = {\bf{e}}'_0$ instead, then the
  resulting array whose $n^{th}$ row is ${\bf p}_n = T^n{\bf p}_0$ is
  a periodization of what has been called the {\it trinomial
    triangle}~\cite{TrinomialTriangle}.
\end{rmk}

Now in the corridor $ \N \times \{ 1, 2, \dots, d - 1 \} $ let $c'_n$
be the number of paths of length $n$ (for simplicity, beginning at $
(0,1))$, but now allowing three choices in movement, up-and-right,
down-and-right and right. With $ {\bf{v}}_n $ defined in the periodic
dual corridor, exactly as above, we have $ {\bf{v}}_{n+1} = (L + I +
R) {\bf{v}}_n $.  With no extra effort, one can conclude that maximum
and minimum values on $ {\bf{p}}_n $ occur on the diagonals $ p_{n,n}
$ and $ p_{n, n + d} $ respectively, and that the difference is our
new corridor number $ c'_n $.  To this point, the proofs are exactly
the same as above; there is just no precursor array of period $d$.
Now let us use this fact to find these corridor numbers within
Pascal's triangle itself, {\it i.e.} to obtain a binomial coefficient
based formula for $c'_n$.

The key is to understand the effect of $ T^n $ on $ {\bf{e}}'_0 = (
\dots, 0, {\bf{1}}, 0 , \dots , 1, 0, \dots ) $, the $2d$-periodic
``unit vector", then to use the linearity of $T^n$.  Using the
commutativity of $R$ and $I + R$, we may apply the binomial theorem to
write
\begin{equation}\label{eqn.T^n}
  T^n = \left[ I + R(I +R) \right]^n = \sum_{j = 0}^{n} \binom{n}{j}
  \sum_{\ell = 0}^{j} \binom{j}{\ell} R^{j + \ell}.
\end{equation}

\begin{theorem}\label{thm.threeway}
  Let $ d \geq 2 $. Then $ p_{n , k} = \sum_{j = 0}^{n} \binom{n}{j}
  \sum_{m \in \Z} \binom{j+1}{2dm - j + k}$.
\end{theorem}
\begin{proof}
  Let us represent $ T^n ( {\bf{e}}'_0 ) $ and $ T^n ( R {\bf{e}}'_0 )
  $ by $ (a_k)_{k \in \N} $ and $ (b_k)_{k \in \N} $ respectively; of
  course these sequences must also be $2d$-periodic. The nonzero
  contributions of $ {\bf{e}}'_0 $ to $ a_k $ come from the $ 1 $'s in
  $ {\bf{e}}'_0 $ occurring at the integer multiples of $2d$ that are
  shifted by $R^{j + l}$ to position $k$ (mod $2d$), {\it i.e.} if $j+
  \ell = 2dm + k $ for $ m \in \Z $. Thus, using~(\ref{eqn.T^n}),
  \[
    a_k = \sum_{j = 0}^{n} \binom{n}{j} \sum_{m \in \Z} \binom{j} {2dm
      - j + k}, \quad \textrm{and} \quad b_k = \sum_{j = 0}^{n}
    \binom{n}{j} \sum_{m \in \Z} \binom{j} {2dm - j + k - 1}.
  \]
  Since $ p_{n,k} = T^n{\bf e}'_0 + T^nR{\bf e}'_0 = a_k + b_k $ we
  immediately obtain the result.
\end{proof}
As a consequence of Theorem~\ref{thm.threeway}, we find
\[
  c'_n = p_{n,n} - p_{n, n+d} = \sum_{j = 0}^{n} \binom{n}{j} \sum_{m
    \in \Z} \left[ \binom{j+1}{2dm - j + n} - \binom{j+1} {2dm - j + n
      + d} \right].
\]
Other counts associated with the three-way corridor paths may also be
analyzed along these lines.

\section{Conclusion}

Early on in this project we used MAPLE to animate plots of a damped
version of $\boldsymbol\sigma_{n}$, for $ n = 0, 1, \dots N$, to
observe a ``wave'' determined by the maximum and minimum values moving
along the successive rows of the Pascal arrays. Upon the up-sampling,
the extreme values fell nicely on diagonals of the circular Pascal
arrays.  Finally, with the introduction of the dual corridor, we found
the strong connection between the two structures.  Using the most
basic properties of a few simple operators, we arrived at our main
result, which easily led to a few nontrivial lattice path results.

\bigskip

The authors acknowledge support from the Department of Mathematics and
Computer Science at Valdosta State University.  We would like to thank
Glen Whitehead and Katie Klimko for their interest and insight into
this problem.  We also thank the anonymous reviewers whose effort
improved the presentation of this paper.


\begin{thebibliography}{10}

\bibitem{Brualdi}
Richard~A. Brualdi.
\newblock {\em Introductory Combinatorics}.
\newblock Pearson, fifth edition, 2008.

\bibitem{KK2011}
C.~Kicey and K.~Klimko.
\newblock Some geometry of {P}ascal's triangle.
\newblock {\em Pi Mu Epsilon Journal}, 13(4):229--245, 2011.

\bibitem{KraMoh}
C.~Krattenthaler and S.~G. Mohanty.
\newblock Lattice path combinatorics -- applications to probability and
  statistics.
\newblock In Norman~L. Johnson, Campbell~B. Read, N.~Balakrishnan, and Brani
  Vidakovic, editors, {\em Encyclopedia of Statistical Sciences}. Wiley, New
  York, second edition, 2003.

\bibitem{Leeuwen2010}
Marc A.~A. Leeuwen.
\newblock Some simple bijections involving lattice walks and ballot sequences.
\newblock Preprint. Available electronically as \url{arXiv:1010.4847}.

\bibitem{Mohanty}
S.~G. Mohanty.
\newblock {\em Lattice Path Counting and Applications}.
\newblock Academic Press, New York, 1979.

\bibitem{Narayana}
T.~V. Narayana.
\newblock {\em Lattice Path Combinatorics with Statistical Applications}.
\newblock Toronto University Press, Toronto, 1979.

\bibitem{OEIS}
N.~J.~A. Sloane.
\newblock The {O}n-{L}ine {E}ncyclopedia of {I}nteger {S}equences.
\newblock published electronically at http://oeis.org/.

\bibitem{Strang}
G.~Strang and T.~Ngyuyen.
\newblock {\em Wavelets and Fliter Banks}.
\newblock Wellesley-Cambridge Press, Wellesley, MA, 1996.

\bibitem{Takacs}
Lajos Tak\'acs.
\newblock Ballot problems.
\newblock {\em Zeitschrift f\"ur Wahrscheinlichkeitstheorie und Verwandte
  Gebiete}, 1(2):154--158, 1962.

\bibitem{TrinomialTriangle}
Eric~W. Weisstein.
\newblock Trinomial triangle.
\newblock From MathWorld--A Wolfram Web Resource.
  \url{http://mathworld.wolfram.com/TrinomialTriangle.html}.

\end{thebibliography}

\bibliographystyle{plain}

\end{document}